\documentclass[11pt,reqno]{amsart}
\usepackage{amsmath,amssymb,amsthm,tinywncy,amscd,graphicx}

\calclayout

\theoremstyle{plain}
\newtheorem{thm}{Theorem}[section]
\newtheorem{theorem}[thm]{Theorem}
\newtheorem{lemma}[thm]{Lemma}

\newtheorem{proposition}[thm]{Proposition}

\theoremstyle{definition}

\newtheorem{definition}[thm]{Definition}
\newtheorem{remark}[thm]{Remark}

\title{The $\boldsymbol{p}$-adic duality for the finite star-multiple polylogarithms}
\author{Shin-ichiro Seki}
\address{Mathematical Institute, Tohoku University, 6-3, Aoba, Aramaki, Aoba-Ku, 
Sendai, 980-8578, Japan}
\email{shinichiro.seki.b3@tohoku.ac.jp}
\thanks{Partly supported by the Grant-in-Aid for JSPS Fellows (JP16J01758), The Ministry of Education, Culture, Sports, Science and Technology, Japan.}
\subjclass[2010]{Primary 11M32.}
\keywords{Finite multiple zeta values, finite multiple polylogarithms.}

\begin{document}
	
\maketitle
	
\begin{abstract}
We prove the $\boldsymbol{p}$-adic duality theorem for the finite star-multiple polylogarithms. That is a generalization of Hoffman's duality theorem for the finite multiple zeta-star values.
\end{abstract}
	\section{Introduction}
	\label{sec:Introduction}
	We begin with the duality for the finite multiple zeta(-star) values (FMZ(S)Vs) in Subsection \ref{subsec:Duality for FMZVs}. Next, we explain the duality for the finite star-multiple polylogarithms (FSMPs) in Subsection \ref{subsec:Duality for FSMPs}. Our main results are Theorem \ref{p-adic duality}, Theorem \ref{dual fn eq for 1-var}, and Theorem \ref{main cor}. The first two theorems are special cases of Theorem \ref{main cor}.
	\subsection{Duality for FMZVs}
	\label{subsec:Duality for FMZVs}
	For any positive integer $n$ and an index $\mathbf{k} = (k_1, \dots, k_m)$, we define the truncated multiple harmonic sums $\zeta_n(\mathbf{k})$ and $\zeta_n^{\star}(\mathbf{k})$ by 
	\begin{align*}
	&\zeta_n(\mathbf{k}):=\sum_{n \geq n_1> \cdots > n_m \geq 1}\frac{1}{n_1^{k_1}\cdots n_m^{k_m}}, \\
	&\zeta_n^{\star}(\mathbf{k}):=\sum_{n \geq n_1\geq \cdots \geq n_m \geq 1}\frac{1}{n_1^{k_1}\cdots n_m^{k_m}},
	\end{align*}
	respectively (we define $\zeta_n(\mathbf{k})$ as $0$ for an empty summation). Then, the multiple zeta value (MZV) $\zeta (\mathbf{k})$ and the multiple zeta-star value (MZSV) $\zeta^{\star}(\mathbf{k})$ are defined by $\zeta (\mathbf{k}):=\lim_{n \to \infty}\zeta_n(\mathbf{k})$ and $\zeta^{\star}(\mathbf{k}):=\lim_{n \to \infty}\zeta^{\star}_n(\mathbf{k})$, respectively when $k_1 \geq 2$. The duality theorem for MZVs $\zeta (\mathbf{k})=\zeta (\mathbf{k}')$ was conjectured firstly by Hoffman in \cite{Ho92} and proved by using the iterated integral (cf.~\cite{Za}). Dualities for MZSVs are not found except a few cases (cf.~\cite{KO}). 
	
	Recently, Kaneko and Zagier \cite{KZ} introduced the {\em finite} multiple zeta values (FMZVs) and several people are studying relations among FMZVs. The FMZV $\zeta^{}_{\mathcal{A}}(\mathbf{k})$ and the finite multiple zeta-star value (FMZSV) $\zeta_{\mathcal{A}}^{\star}(\mathbf{k})$ are defined by $\zeta^{}_{\mathcal{A}}(\mathbf{k}) := (\zeta_{p-1}(\mathbf{k})\bmod{p})_p$ and $\zeta_{\mathcal{A}}^{\star}(\mathbf{k}) := (\zeta_{p-1}^{\star}(\mathbf{k})\bmod{p})_p$ respectively in the $\mathbb{Q}$-algebra $\mathcal{A} = \left. \left( \prod_p\mathbb{Z}/p\mathbb{Z} \right) \right/ \left( \bigoplus_p\mathbb{Z}/p\mathbb{Z} \right)$, where $p$ runs over all prime numbers.
	Around 2000, the duality theorem for FMZSVs was discovered and proved by Hoffman \cite[Theorem 4.6]{Ho}:
	\begin{equation}
	\zeta_{\mathcal{A}}^{\star}(\mathbf{k}) = -\zeta_{\mathcal{A}}^{\star}(\mathbf{k}^{\vee})
	\label{Hoffman duality},\end{equation}
	where $\mathbf{k}^{\vee}$ is the Hoffman dual of the index $\mathbf{k}$ (See Definition \ref{Hoffman dual}). This is a counterpart to the duality theorem for MZVs. Comparing with the duality of MZVs, it is worth mentioning that such a simple duality is satisfied by FMZSVs rather than by non-star FMZVs. The duality (\ref{Hoffman duality}) is one of basic relations among FMZ(S)Vs and some other proofs are given by Imatomi \cite[Corollary 4.1]{I} and Yamamoto \cite[p.~3]{Y}. 
	
	In order to rewrite the duality (\ref{Hoffman duality}) to relations for non-star FMZVs, let us recall terminologies of Hoffman's algebra.
	Let $\mathfrak{H}:=\mathbb{Q}\langle x, y\rangle$ be a non-commutative polynomial algebra in two variables and $\mathfrak{H}^1:= \mathbb{Q}+\mathfrak{H}y$ its subalgebra. Let $z_k:=x^{k-1}y$ for a natural number $k$ and $z_{\mathbf{k}}:=z_{k_1}\cdots z_{k_m}$ for an index $\mathbf{k}=(k_1, \dots, k_m)$. Then, $\mathfrak{H}^1$ is generated by $z_k \ (k=1, 2, \dots)$ as a non-commutative algebra. We define the $\mathbb{Q}$-linear map $Z_{\mathcal{A}}\colon \mathfrak{H}^1 \to \mathcal{A}$ characterized by $Z_{\mathcal{A}}(1)=1$ and $Z_{\mathcal{A}}(z_{\mathbf{k}})=\zeta^{}_{\mathcal{A}}(\mathbf{k})$ for each index $\mathbf{k}$. We also define an algebra automorphism $\psi \colon \mathfrak{H}^1 \to \mathfrak{H}^1$ by $x\mapsto x+y$ and $y \mapsto -y$. In this setup, Hoffman proved that the duality theorem (\ref{Hoffman duality}) is equivalent to the following relations for FMZVs:
	\begin{theorem}[{\cite[Theorem 4.7]{Ho}}]
		For any word $w \in \mathfrak{H}^1$, we have
		\begin{equation*}
		\psi (w) - w \in \mathrm{ker}(Z_{\mathcal{A}}).
		\end{equation*}
		\label{equivalent theorem}\end{theorem}
	Next, we recall the $\mathbb{Q}$-algebra $\widehat{\mathcal{A}}$ introduced by Rosen \cite{Ro}. For any positive integer $n$, we define $\mathcal{A}_n$ to be the quotient ring $\left. \left( \prod_p\mathbb{Z}/p^n\mathbb{Z} \right) \right/ \left( \bigoplus_p \mathbb{Z}/p^n\mathbb{Z} \right)$. Then, the rings $\{\mathcal{A}_n\}$ becomes a projective system by natural projections and we define $\widehat{\mathcal{A}}$ to be the projective limit $\varprojlim_n \mathcal{A}_n$. We equip $\mathcal{A}_n$ with the discrete topology for each $n$ and $\widehat{\mathcal{A}}$ with the projective limit topology. The $\mathbb{Q}$-algebra $\widehat{\mathcal{A}}$ is complete and not locally compact. There exist natural projections $\pi \colon \widehat{\mathbb{Z}}=\prod_p\mathbb{Z}_p \twoheadrightarrow \widehat{\mathcal{A}}$ and $\pi_n\colon \widehat{\mathcal{A}} \twoheadrightarrow \mathcal{A}_n$ for any $n$, where $\mathbb{Z}_p$ is the ring of $p$-adic integers. We redefine the FMZV $\zeta_{\widehat{\mathcal{A}}}(\mathbf{k})$ and the FMZSV $\zeta_{\widehat{\mathcal{A}}}^{\star}(\mathbf{k})$ to be $\pi ((\zeta_{p-1}(\mathbf{k}))_p)$ and $\pi ((\zeta_{p-1}^{\star}(\mathbf{k}))_p)$, respectively as elements of $\widehat{\mathcal{A}}$. Furthermore, $\zeta^{}_{\mathcal{A}_n}(\mathbf{k}) := \pi_n(\zeta_{\widehat{\mathcal{A}}}(\mathbf{k}))$ and $\zeta_{\mathcal{A}_n}^{\star}(\mathbf{k}) := \pi_n(\zeta_{\widehat{\mathcal{A}}}^{\star}(\mathbf{k}))$ in $\mathcal{A}_n$. We define the element $\boldsymbol{p} := \pi((p)_p) \in \widehat{\mathcal{A}}$ and we also denote $\pi_n(\boldsymbol{p}) \in \mathcal{A}_n$ by $\boldsymbol{p}$ by abuse of notation. We can check that the topology of $\widehat{\mathcal{A}}$ is the $\boldsymbol{p}$-adic topology (See Subsection \ref{subsec:adelic ring}).
	
	Let $\widehat{\mathfrak{H}}^1$ be the completion of $\mathfrak{H}^1$. Namely, $\widehat{\mathfrak{H}}$ is defined as the non-commutative formal power series ring $\mathbb{Q}\langle \! \langle x, y \rangle \! \rangle$ and $\widehat{\mathfrak{H}}^1:=\mathbb{Q}+\widehat{\mathfrak{H}}y$. Then, the weighted finite multiple zeta function $Z_{\widehat{\mathcal{A}}}\colon \widehat{\mathfrak{H}}^1 \to \widehat{\mathcal{A}}$ is defined by
	\[
	\sum_{\mathbf{k}}a_{\mathbf{k}}z_{\mathbf{k}} \mapsto \sum_{\mathbf{k}}a_{\mathbf{k}}\zeta_{\widehat{\mathcal{A}}}(\mathbf{k})\boldsymbol{p}^{\mathrm{wt}(\mathbf{k})},
	\]
	where $a_{\mathbf{k}} \in \mathbb{Q}$ and $\mathrm{wt}(\mathbf{k})$ is the weight of the index $\mathbf{k}$. The algebra automorphism $\psi$ on $\mathfrak{H}^1$ is extended continuously to the map on $\widehat{\mathfrak{H}}^1$ and we define a continuous algebra automorphism $\Phi\colon \widehat{\mathfrak{H}}^1 \to \widehat{\mathfrak{H}}^1$ by
	\[
	w \mapsto (1+y)\left( \frac{1}{1+y}\ast w \right).
	\]
	Here, the harmonic product $\ast \colon \mathfrak{H}^1\times \mathfrak{H}^1 \to \mathfrak{H}^1$ is defined $\mathbb{Q}$-bilinearly and inductively by
	\begin{equation*}
	w\ast 1=1\ast w = w, \ z_kw_1\ast z_lw_2 = z_k(w_1\ast z_lw_2)+z_l(z_kw_1\ast w_2)+z_{k+l}(w_1\ast w_2)
	\end{equation*}
	for any positive integers $k, l$ and words $w, w_1, w_2 \in \mathfrak{H}^1$ and $\ast$ is extended naturally to the product on $\widehat{\mathfrak{H}}^1$. Rosen generalized Theorem \ref{equivalent theorem} as follows:
	\begin{theorem}[{Asymptotic duality theorem \cite[Theorem 4.5]{Ro}}]
		For any $w \in \widehat{\mathfrak{H}}^1$, we have
		\begin{equation*}
		\psi(w)-\Phi (w) \in \mathrm{ker}(Z_{\widehat{\mathcal{A}}}).
		\end{equation*}
	\end{theorem}
	On the other hand, Zhao, Sakugawa, and the author proved the straightforward generalization of the duality (\ref{Hoffman duality}) to an $\mathcal{A}_2$-relation (\cite[Theorem 2.11]{Z} and \cite[The equality (40)]{SS}). We can rewrite the relation as the following symmetric form:
	\begin{equation*}
	\zeta_{\mathcal{A}_2}^{\star}(\mathbf{k})+\zeta_{\mathcal{A}_2}^{\star}(1, \mathbf{k})\boldsymbol{p} = -\zeta_{\mathcal{A}_2}^{\star}(\mathbf{k}^{\vee})-\zeta_{\mathcal{A}_2}^{\star}(1, \mathbf{k}^{\vee})\boldsymbol{p}.
	\label{A_2-duality}\end{equation*}
	In this paper, we give the following $\boldsymbol{p}$-adic version of the duality (\ref{Hoffman duality}):
	\begin{theorem}[The $\boldsymbol{p}$-adic duality theorem for FMZSVs]
		Let $\mathbf{k}$ be an index. Then, we have
		\begin{equation*}
		\sum_{i=0}^{\infty}\zeta_{\widehat{\mathcal{A}}}^{\star}(\{1\}^i, \mathbf{k})\boldsymbol{p}^i = - \sum_{i=0}^{\infty}\zeta_{\widehat{\mathcal{A}}}^{\star}(\{1\}^i, \mathbf{k}^{\vee})\boldsymbol{p}^i
		\end{equation*}
		in the ring $\widehat{\mathcal{A}}$.
		\label{p-adic duality}\end{theorem}
	\noindent Here, the notation $(\{1\}^i, \mathbf{k})$ means $(\underbrace{1, \dots, 1}_i, k_1, \dots, k_m)$ for $\mathbf{k}=(k_1, \dots, k_m)$. We remark that if we take the $i=0$ part of the equality in the above theorem, then we recover the equality (\ref{Hoffman duality}).
	\subsection{Duality for FSMPs}
	\label{subsec:Duality for FSMPs}
	In \cite{SS}, Sakugawa and the author introduced the finite (star-) multiple polylogarithms (F(S)MPs) and proved the following dual functional equation of FSMPs which is a generalization of Hoffman's duality theorem (\ref{Hoffman duality}):
	\begin{theorem}[{Sakugawa-Seki \cite[Theorem 1.3]{SS}}]
		Let $\mathbf{k}$ be an index. Then, we have
		\[
		\widetilde{\text{\rm \pounds}}_{\mathcal{A}, \mathbf{k}}^{\star}(t)-\frac{1}{2}\zeta_{\mathcal{A}}^{\star}(\mathbf{k}) = \widetilde{\text{\rm \pounds}}_{\mathcal{A}, \mathbf{k}^{\vee}}^{\star}(1-t)-\frac{1}{2}\zeta_{\mathcal{A}}^{\star}(\mathbf{k}^{\vee})
		\]
		in the ring $\mathcal{A}_{\mathbb{Z}[t]} = \left. \left( \prod_p \mathbb{Z}/p\mathbb{Z}[t] \right) \right/ \left( \bigoplus_p \mathbb{Z}/p\mathbb{Z}[t] \right)$.
		\label{SS thm}\end{theorem}
	\noindent See Definition \ref{def of FMP} for the definition of FMPs.
	
	In this paper, we prove the following $\boldsymbol{p}$-adic version of Theorem \ref{SS thm} which contains Theorem \ref{p-adic duality} as a special case:
	\begin{theorem}
		Let $\mathbf{k}$ be an index. Then, we have
		{\small \begin{equation}
			\sum_{i=0}^{\infty}\left( \widetilde{\text{\rm \pounds}}_{\widehat{\mathcal{A}}, (\{1\}^i, \mathbf{k})}^{\star}(t)-\frac{1}{2}\zeta_{\widehat{\mathcal{A}}}^{\star}(\{1\}^i, \mathbf{k})\right) \boldsymbol{p}^i = \sum_{i=0}^{\infty}\left( \widetilde{\text{\rm \pounds}}_{\widehat{\mathcal{A}}, (\{1\}^i, \mathbf{k}^{\vee})}^{\star}(1-t)-\frac{1}{2}\zeta_{\widehat{\mathcal{A}}}^{\star}(\{1\}^i, \mathbf{k}^{\vee})\right) \boldsymbol{p}^i.
			\label{1-var}\end{equation}
		}in the ring $\widehat{\mathcal{A}}_{\mathbb{Z}[t]} = \varprojlim_n \left. \left( \prod_p \mathbb{Z}/p^n\mathbb{Z}[t] \right) \right/ \left( \bigoplus_p\mathbb{Z}/p^n\mathbb{Z}[t] \right)$.
		\label{dual fn eq for 1-var}\end{theorem}
	We remark that if we take the $i=0$ part of the equality (\ref{1-var}), then we recover Theorem \ref{SS thm}. More generally, Sakugawa and the author proved the multi-variable and $\mathcal{A}_2$-version of Theorem \ref{SS thm} (\cite[Theorem 3.12]{SS}) and the main result of this paper is the $\boldsymbol{p}$-adic dual functional equation for the multi-variable FSMPs (= Theorem \ref{main cor}) which contains \cite[Theorem 3.12]{SS} and Theorem \ref{dual fn eq for 1-var} as special cases.
	
	This paper is organized as follows. In Section \ref{sec:Definitions}, we prepare some notation and define the finite multiple polylogarithms. In Section \ref{sec:reversal and duality}, we prove the $\boldsymbol{p}$-adic reversal theorem for FMPs and state the $\boldsymbol{p}$-adic duality theorem for FSMPs. In Section \ref{sec:Proof of main result}, we complete the proof of main results. 
	
	\section*{Acknowledgement}
	The author would like to express his sincere gratitude to his advisor Professor Tadashi Ochiai for carefully reading the manuscript and helpful comments. The author also thanks Dr.~Kenji Sakugawa for useful discussion and helpful advice. In addition, the author would like to thank the anonymous referee for pointing out several errors and useful suggestions. The proof of Proposition \ref{key proposition} was greatly shortened by his/her idea of using Lemma \ref{lemma by the referee}, though the author's original proof was based on more complicated calculations.
	\section{Notation and Definitions}
	\label{sec:Definitions}
	For a tuple of indeterminates $\mathbf{t} = (t_1, \dots, t_m)$, we define $\mathbf{t}_1$, $1-\mathbf{t}$, $\mathbf{t}^{-1}$, and $\overline{\mathbf{t}}$ to be $(t_1, \dots, t_{m-1}, 1)$, $(1-t_1, \dots, 1-t_m)$, $(t_1^{-1}, \dots, t_m^{-1})$, and $(t_m, \dots, t_1)$ respectively. We use the notation $R[\mathbf{t}]$ as a polynomial ring $R[t_1, \dots, t_m]$ for a ring $R$. 
	\subsection{Indices}
	\label{subsec:Indices}
	We call a tuple of positive integers $\mathbf{k} = (k_1, \dots, k_m)$ an index and we define the weight $\mathrm{wt}(\mathbf{k})$ (resp. depth $\mathrm{dep}(\mathbf{k})$) of $\mathbf{k}$ to be $k_1+\cdots +k_m$ (resp. $m$). 
	
	Let $\mathbf{k}=(k_1, k_2, \dots, k_m)$, $\mathbf{k}'=(k'_1, \dots, k'_m)$, and $\mathbf{l}=(l_1, \dots, l_n)$ be indices. Then, we define the reverse index $\overline{\mathbf{k}}$ of $\mathbf{k}$, the summation $\mathbf{k}\oplus\mathbf{k'}$, and the concatenation index $(\mathbf{k}, \mathbf{l})$ by
	\begin{align*}
	\overline{\mathbf{k}}&:=(k_m, \dots, k_1),\\
	\mathbf{k}\oplus \mathbf{k'} &:=(k_1+k'_1, \dots, k_m+k'_m),\\
	(\mathbf{k}, \mathbf{l}) &:= (k_1, \dots, k_m, l_1, \dots, l_n),
	\end{align*}
	respectively. We use the same notation for tuples of non-negative integers or indeterminates.
	
	Let $W$ be the free monoid generated by the set $\{0, 1\}$. We set $W_1 := W1$. Then, there exists a bijection from the set of all indices $I$ to $W_1$ induced by the correspondence
	\[
	\mathbf{k} = (k_1, \dots, k_m) \mapsto \underbrace{0 \cdots 0}_{k_1-1}1\underbrace{0 \cdots 0}_{k_2-1}1 \cdots 1\underbrace{0 \cdots 0}_{k_m-1}1
	\] 
	and we denote the bijection as $w$.
	Let $\tau \colon W \to W$ be a monoid homomorphism defined by $\tau(0)=1$ and $\tau(1)=0$.
	\begin{definition}[{cf. \cite[Section3]{Ho}}]
		For an index $\mathbf{k}$, we define the Hoffman dual $\mathbf{k}^{\vee}$ of $\mathbf{k}$ by the relation $w(\mathbf{k}^{\vee})=\tau(w(\mathbf{k})1^{-1})1$.
		\label{Hoffman dual}\end{definition}
	\noindent For any index $\mathbf{k}$, $\mathrm{wt}(\mathbf{k})=\mathrm{wt}(\mathbf{k}^{\vee})$ and $\mathrm{dep}(\mathbf{k})+\mathrm{dep}(\mathbf{k}^{\vee}) = \mathrm{wt}(\mathbf{k})+1$ hold.
	\subsection{The adelic ring}
	\label{subsec:adelic ring}
	In order to define the finite multiple polylogarithms, we introduce some adelic rings in a general setting.
	\begin{definition}
		Let $R$ be a commutative ring and $\Sigma$ an infinite family of ideals of $R$. Then, we define the ring $\mathcal{A}_{n, R}^{\Sigma}$ for each positive integer $n$ by 
		\begin{equation*}
		\mathcal{A}_{n, R}^{\Sigma} := \left( \prod_{I \in \Sigma}R/I^n \right) \left/ \left( \bigoplus_{I \in \Sigma}R/I^n\right) \right. .
		\end{equation*}
		Then, $\{\mathcal{A}_{n, R}^{\Sigma}\}$ becomes a projective system by natural projections and we define the ring $\widehat{\mathcal{A}}_{R}^{\Sigma}$ by 
		\begin{equation*}
		\widehat{\mathcal{A}}_R^{\Sigma} := \varprojlim_{n}\mathcal{A}_{n, R}^{\Sigma}.
		\end{equation*}
		We put the discrete topology on $\mathcal{A}_{n, R}^{\Sigma}$ for each $n$ and we define the topology of $\widehat{\mathcal{A}}_{R}^{\Sigma}$ to be the projective limit topology. 
		\label{def of adelic}\end{definition}
	\begin{lemma}
		We use the same notation as in Definition $\ref{def of adelic}$ and we define the $I$-adic completion $\widehat{R}_I$ of $R$ to be $\varprojlim_n R/I^nR$. Then, there exists the following natural surjective ring homomorphism$:$
		\[
		\pi\colon \prod_{I \in \Sigma}\widehat{R}_I \longrightarrow \widehat{\mathcal{A}}_R^{\Sigma}.
		\]
		\label{surjection}\end{lemma}
	\begin{proof}
		For a short exact sequence of projective systems of rings
		\[
		0 \longrightarrow \left\{ \bigoplus_{I \in \Sigma}R/I^n \right\} \longrightarrow \left\{ \prod_{I \in \Sigma}R/I^n \right\} \longrightarrow \left\{ \mathcal{A}_{n, R}^{\Sigma} \right\} \longrightarrow 0,
		\]
		the system $\left\{ \bigoplus_{I \in \Sigma}R/I^n \right\}$ satisfies the Mittag-Leffler condition. Therefore, there exists a natural surjection
		\[
		\prod_{I \in \Sigma}\widehat{R}_I \simeq \varprojlim_n \prod_{I \in \Sigma}R/I^n \longrightarrow \widehat{\mathcal{A}}_R^{\Sigma}.
		\]
	\end{proof}
	\begin{remark}
		We assume that some topology of $R/I^n$ is defined for any $I \in \Sigma$. If we put the product topology on $\prod_{I \in \Sigma}R/I^n$ and the quotient topology on $\mathcal{A}_{n, R}^{\Sigma}$ by $\prod_{I \in \Sigma}R/I^n \twoheadrightarrow \mathcal{A}_{n, R}^{\Sigma}$, then the topology is indiscrete. However, we consider the discrete topology of $\mathcal{A}_{n, R}^{\Sigma}$ in this paper.
	\end{remark}
	\begin{lemma}
		We use the same notation as in Definition $\ref{def of adelic}$ and Definition $\ref{surjection}$. We assume that $I\widehat{R}_I$ is a principal ideal of $\widehat{R}_I$ for any $I \in \Sigma$. Furthermore, we define an ideal $\boldsymbol{I}$ of $\widehat{\mathcal{A}}_{R}^{\Sigma}$ to be $\pi ((I\widehat{R}_I)_{I \in \Sigma})$. Let $\pi_n$ be the  natural projection $\pi_n\colon \widehat{\mathcal{A}}_{R}^{\Sigma} \twoheadrightarrow \mathcal{A}_{n, R}^{\Sigma}$ for any positive integer $n$. Then, we have $\mathrm{ker}(\pi_n) = \boldsymbol{I}^n$. In particular, the topology of $\widehat{\mathcal{A}}_R^{\Sigma}$ coincides with the $\boldsymbol{I}$-adic topology and $\widehat{\mathcal{A}}_R^{\Sigma}$ is complete with respect to the $\boldsymbol{I}$-adic topology.
		\label{I-adic}\end{lemma}
	\begin{proof}
		Let $n$ be a positive integer. Take any element $x$ of $\mathrm{ker}(\pi_n)$. Then, there exists an element $\{x_I\}_{I \in \Sigma}$ of $\prod_{I \in \Sigma}\widehat{R}_I$ such that $x=\pi ((x_I)_{I \in \Sigma})$ by Lemma \ref{surjection}. By the commutative diagram
		\begin{equation*}
		\begin{CD}
		\displaystyle \prod_{I \in \Sigma}\widehat{R}_I @>\pi>> \widehat{\mathcal{A}}_R^{\Sigma} \\
		@V(\bmod{I^n})_{I \in \Sigma}VV @VV\pi_nV \\
		\displaystyle \prod_{I \in \Sigma}R/I^n @>\rho_n>> \mathcal{A}_{n, R}^{\Sigma}
		\end{CD}
		\end{equation*}
		we have
		\[
		\pi_n(x)=\pi_n \circ \pi ((x_I)_{I \in \Sigma})=\rho_n((x_I \bmod{I^n})_{I \in \Sigma}) = 0.
		\]
		Here, $\rho_n$ is the canonical projection. Therefore, there exists a subset $\Sigma'$ of $\Sigma$ such that $\Sigma \setminus \Sigma'$ is finite and $x_I \in I^n\widehat{R}_I$ for every $I \in \Sigma'$. We can take a generator $a_I$ of $I\widehat{R}_I$ for any $I \in \Sigma$ by the assumption. Then, there exists an element $\{y_I\}_{I \in \Sigma'}$ of $\prod_{I \in \Sigma'}\widehat{R}_I$ such that $x_I=a_I^ny_I$ holds for any $I \in \Sigma'$. We define $y_I$ to be zero for $I \in \Sigma \setminus \Sigma'$. Then, we have
		\[
		x=\pi ((x_I)_{I \in \Sigma}) = \pi ((a_I^ny_I)_{I \in \Sigma}) = (\pi((a_I)_{I \in \Sigma}))^n\cdot \pi ((y_I)_{I \in \Sigma}) \in \boldsymbol{I}^n
		\]
		and we obtain the inclusion $\mathrm{ker}(\pi_n) \subset \boldsymbol{I}^n$. The opposite inclusion is trivial and the last assertion follows from the fact that $\{ \mathrm{ker}(\pi_n)\}$ is a neighborhood basis of zero.
	\end{proof}
	In the rest of this paper, we only use the case $\Sigma = \{ pR \mid p \ \text{is a prime number}\}$ and we omit the notation $\Sigma$. We will define the finite multiple polylogarithms as elements of the $\mathbb{Q}$-algebra $\widehat{\mathcal{A}}_{\mathbb{Z}[\mathbf{t}]}$ in the next subsection. Let $\pi\colon \prod_p\widehat{\mathbb{Z}[\mathbf{t}]}_p \twoheadrightarrow \widehat{\mathcal{A}}_{\mathbb{Z}[\mathbf{t}]}$ be the natural surjection obtained by Lemma \ref{surjection} where $\widehat{\mathbb{Z}[\mathbf{t}]}_p = \varprojlim_n\mathbb{Z}[\mathbf{t}]/p^n\mathbb{Z}[\mathbf{t}]$ is the $p$-adic completion of $\mathbb{Z}[\mathbf{t}]$. Let $\pi_n\colon \widehat{\mathcal{A}}_{\mathbb{Z}[\mathbf{t}]} \twoheadrightarrow \mathcal{A}_{n, \mathbb{Z}[\mathbf{t}]}$ be the natural projection for each $n$. The topology of $\widehat{\mathcal{A}}_{\mathbb{Z}[\mathbf{t}]}$ coincides with the $\boldsymbol{p}$-adic topology and $\widehat{\mathcal{A}}_{\mathbb{Z}[\mathbf{t}]}$ is complete with respect to the topology by Lemma \ref{I-adic}. Since an equality $\pi \left( \left( \sum_{i=0}^{\infty}a_i^{(p)}p^i\right)_p \right) = \sum_{i=0}^{\infty}(a_i^{(p)})_p\boldsymbol{p}^i$ holds, in order to obtain a $\boldsymbol{p}$-adic relation, it is sufficient to show the $p$-adic relations given by taking the $p$-components for all but finitely many prime numbers $p$. Here, $a_i^{(p)} \in \mathbb{Z}_{(p)}[\mathbf{t}]$ and $\mathbb{Z}_{(p)}$ is the localization of $\mathbb{Z}$ at $p$, and note that the opposite assertion does not hold in general.
	\subsection{The finite multiple polylogarithms}
	\label{subsec:The finite multiple polylogaarithms}
	\begin{definition}
		Let $n$ be a positive integer, $\mathbf{k} = (k_1, \dots , k_m)$ an index, and $\mathbf{t} = (t_1, \dots, t_m)$ a tuple of indeterminates. Then, we define the four kinds of the truncated multiple polylogarithms which are elements of $\mathbb{Q}[\mathbf{t}]$ as follows:
		\begin{align*}
		\text{\rm \pounds}_{n, \mathbf{k}}^{\ast}(\mathbf{t}) &:= \sum_{n\geq n_1> \dots >n_m\geq 1}\frac{t_1^{n_1}\cdots t_m^{n_m}}{n_1^{k_1}\cdots n_m^{k_m}},\\
		\text{\rm \pounds}_{n, \mathbf{k}}^{\ast , \star}(\mathbf{t}) &:= \sum_{n\geq n_1\geq \dots \geq n_m\geq 1}\frac{t_1^{n_1}\cdots t_m^{n_m}}{n_1^{k_1}\cdots n_m^{k_m}},\\
		\text{\rm \pounds}_{n,\mathbf{k}}^{\scalebox{0.7}{\text{\cyr sh}}}(\mathbf{t}) &:= \sum_{n \geq n_1> \dots >n_m\geq 1}\frac{t_1^{n_1-n_2}\cdots t_{m-1}^{n_{m-1}-n_m}t_m^{n_m}}{n_1^{k_1}\cdots n_m^{k_m}},\\
		\text{\rm \pounds}_{n, \mathbf{k}}^{\scalebox{0.7}{\text{\cyr sh}}, \star}(\mathbf{t}) &:= \sum_{n\geq n_1\geq \dots \geq n_m\geq 1}\frac{t_1^{n_1-n_2}\cdots t_{m-1}^{n_{m-1}-n_m}t_m^{n_m}}{n_1^{k_1}\cdots n_m^{k_m}}.
		\end{align*}
	\end{definition}
	\begin{definition}
		Let $\mathbf{k} = (k_1, \dots , k_m)$ be an index and $\mathbf{t} = (t_1, \dots, t_m)$ a tuple of indeterminates. Then, we define the four kinds of the finite multiple polylogarithms which are elements of $\widehat{\mathcal{A}}_{\mathbb{Z}[\mathbf{t}]}$ as follows:
		\begin{align*}
		\text{\rm \pounds}_{\widehat{\mathcal{A}}, \mathbf{k}}^{\ast}(\mathbf{t}) &:= \pi ((\text{\rm \pounds}_{p-1, \mathbf{k}}^{\ast}(\mathbf{t}))_p) \ \ (\text{\emph{the finite harmonic multiple polylogarithm}}),\\
		\text{\rm \pounds}_{\widehat{\mathcal{A}}, \mathbf{k}}^{\ast, \star}(\mathbf{t}) &:= \pi ((\text{\rm \pounds}_{p-1, \mathbf{k}}^{\ast, \star}(\mathbf{t}))_p) \ \ (\text{\emph{the finite harmonic star-multiple polylogarithm}}),\\
		\text{\rm \pounds}_{\widehat{\mathcal{A}}, \mathbf{k}}^{\scalebox{0.7}{\text{\cyr sh}}}(\mathbf{t}) &:= \pi ((\text{\rm \pounds}_{p-1, \mathbf{k}}^{\scalebox{0.7}{\text{\cyr sh}}}(\mathbf{t}))_p) \ \ (\text{\emph{the finite shuffle multiple polylogarithm}}),\\
		\text{\rm \pounds}_{\widehat{\mathcal{A}}, \mathbf{k}}^{\scalebox{0.7}{\text{\cyr sh}}, \star}(\mathbf{t}) &:= \pi ((\text{\rm \pounds}_{p-1, \mathbf{k}}^{\scalebox{0.7}{\text{\cyr sh}}, \star}(\mathbf{t}))_p) \ \ (\text{\emph{the finite shuffle star-multiple polylogarithm}}).
		\end{align*}
		This definition is well-defined since $\text{\rm \pounds}_{p-1, \mathbf{k}}^{\circ, \bullet}(\mathbf{t})$ is an element of $\mathbb{Z}_{(p)}[\mathbf{t}]$ for each prime number $p$, $\circ \in \{\ast, \text{{\cyr sh}}\}$, and $\bullet \in \{\emptyset, \star\}$. We also define the finite multiple polylogarithm $\text{\rm \pounds}_{\mathcal{A}_n, \mathbf{k}}^{\circ, \bullet}(\mathbf{t})$ as elements of $\mathcal{A}_{n, \mathbb{Z}[\mathbf{t}]}$ by 
		\[
		\text{\rm \pounds}_{\mathcal{A}_n, \mathbf{k}}^{\circ, \bullet}(\mathbf{t}) := \pi_n(\text{\rm \pounds}_{\widehat{\mathcal{A}}, \mathbf{k}}^{\circ, \bullet}(\mathbf{t}))
		\]
		for each positive integer $n$, $\circ \in \{\ast, \text{{\cyr sh}}\}$, and $\bullet \in \{\emptyset, \star\}$.
		We define 1-variable finite (star-) multiple polylogarithms as follows:
		\begin{equation*}
		\begin{split}
		\text{\rm \pounds}_{\widehat{\mathcal{A}}, \mathbf{k}}^{\bullet}(t) &:= \text{\rm \pounds}_{\widehat{\mathcal{A}}, \mathbf{k}}^{\ast , \bullet} (t, \{ 1 \}^{m-1}) = \text{\rm \pounds}_{\widehat{\mathcal{A}}, \mathbf{k}}^{\scalebox{0.7}{\text{\cyr sh}} , \bullet} (\{ t \}^{m}) \in \widehat{\mathcal{A}}_{\mathbb{Z}[t]},\\
		\widetilde{\text{\rm \pounds}}_{\widehat{\mathcal{A}}, \mathbf{k}}^{\bullet}(t) &:= \text{\rm \pounds}_{\widehat{\mathcal{A}}, \mathbf{k}}^{\ast , \bullet} (\{ 1 \}^{m-1}, t) = \text{\rm \pounds}_{\widehat{\mathcal{A}}, \mathbf{k}}^{\scalebox{0.7}{\text{{\cyr sh}}} , \bullet} (\{ 1 \}^{m-1}, t) \in \widehat{\mathcal{A}}_{\mathbb{Z}[t]},
		\end{split}
		\end{equation*}
		where $t$ is an indeterminate and $\bullet \in \{ \emptyset , \star \}$.
		\label{def of FMP}\end{definition}
	\section{The $\boldsymbol{p}$-adic reversal theorem and the $\boldsymbol{p}$-adic duality theorem}
	\label{sec:reversal and duality}
	\subsection{The $\boldsymbol{p}$-adic reversal theorem for FMPs}
	\label{subsec:Asymptotic reversal theorems}
	The reversal relation for FMZVs (\cite[Theorem 4.5]{Ho}) has been extended to several general cases. For example, Rosen proved the $\boldsymbol{p}$-adic reversal relation for FMZVs in $\widehat{\mathcal{A}}$ (\cite[Theorem 4.1]{Ro}) and Sakugawa and the author proved the reversal relation for FMPs in $\mathcal{A}_{2, \mathbb{Z}[\mathbf{t}]}$ (\cite[Proposition 3.11]{SS}). Here, we prove the $\boldsymbol{p}$-adic reversal relation for FMPs in $\widehat{\mathcal{A}}_{\mathbb{Z}[\mathbf{t}]}$.
	\begin{theorem}Let $\mathbf{k} = (k_1, \dots, k_m)$ be an index, $\mathbf{t}=(t_1, \dots, t_m)$ a tuple of indeterminates, and $\bullet \in \{\emptyset, \star\}$. Then, we have the following $\boldsymbol{p}$-adic relation in $\widehat{\mathcal{A}}_{\mathbb{Z}[\mathbf{t}]}:$
		\begin{equation*}
		\text{\rm \pounds}_{\widehat{\mathcal{A}}, \overline{\mathbf{k}}}^{\ast, \bullet}(\mathbf{t})=(-1)^{\mathrm{wt}(\mathbf{k})}(t_1\cdots t_m)^{\boldsymbol{p}}\sum_{i=0}^{\infty}\sum_{\substack{\mathbf{l}=(l_1, \dots, l_m) \in \mathbb{Z}_{\geq 0}^m \\ \mathrm{wt}(\mathbf{l})=i}}\left[ \prod_{j=1}^m\binom{k_j+l_j-1}{l_j}\right] \text{\rm \pounds}_{\widehat{\mathcal{A}}, \mathbf{k}\oplus \mathbf{l}}^{\ast, \bullet}(\overline{\mathbf{t}^{-1}})\boldsymbol{p}^i,
		\end{equation*}
		where $(t_1\cdots t_m)^{\boldsymbol{p}}=((t_1\cdots t_m)^p)_p \in \widehat{\mathcal{A}}_{\mathbb{Z}[\mathbf{t}]}$ and $\text{\rm \pounds}_{\widehat{\mathcal{A}}, \mathbf{k}\oplus \mathbf{l}}^{\ast, \bullet}(\overline{\mathbf{t}^{-1}})$ is an element of $\widehat{\mathcal{A}}_{\mathbb{Z}[\mathbf{t}^{-1}]}$.
	\end{theorem}
	\begin{proof}
		Let $p$ be a prime number. Since a $p$-adically convergent identity
		\begin{equation*}
		\frac{1}{(p-n)^k}=(-1)^k\sum_{l=0}^{\infty}\binom{k+l-1}{l}\frac{p^l}{n^{k+l}}
		\label{p-adic reverse}\end{equation*}
		holds for a positive integer $n < p$, by the substitutions $n_i \mapsto p-n_{m+1-i}$, we have
		\begin{equation*}
		\begin{split}
		\text{\rm \pounds}_{p-1, \overline{\mathbf{k}}}^{\ast}(\mathbf{t}) &= \sum_{p-1 \geq n_1 > \cdots > n_m \geq 1}\frac{t_1^{n_1}\cdots t_m^{n_m}}{n_1^{k_m}\cdots n_m^{k_1}} \\ &= \sum_{p-1 \geq p-n_m > \cdots > p-n_1 \geq 1} \frac{t_1^{p-n_m}\cdots t_m^{p-n_1}}{(p-n_m)^{k_m}\cdots (p-n_1)^{k_1}} \\ &= (-1)^{\mathrm{wt}(\mathbf{k})}(t_1\cdots t_m)^{p}\\ & \ \ \ \times \sum_{p-1 \geq n_1 > \cdots > n_m \geq 1}\sum_{(l_1, \dots, l_m ) \in \mathbb{Z}^m_{\geq 0}}\left[ \prod_{j=1}^m\binom{k_j+l_j-1}{l_j} \right]\frac{t_m^{-n_1}\cdots t_1^{-n_m}}{n_1^{k_1+l_1}\cdots n_m^{k_m+l_m}}p^{l_1+\cdots +l_m}\\ &= (-1)^{\mathrm{wt}(\mathbf{k})}(t_1\cdots t_m)^p\sum_{i=0}^{\infty}\sum_{\substack{\mathbf{l}=(l_1, \dots, l_m) \in \mathbb{Z}_{\geq 0}^m \\ \mathrm{wt}(\mathbf{l})=i}}\left[ \prod_{j=1}^m \binom{k_j+l_j-1}{l_j} \right] \text{\rm \pounds}_{p-1, \mathbf{k}\oplus \mathbf{l}}^{\ast}(\overline{\mathbf{t}^{-1}})p^i
		\end{split}
		\end{equation*}
		in the ring $\widehat{\mathbb{Z}[\mathbf{t}]}_p$. Therefore, we have the conclusion for non-star case. The star case is similar.
	\end{proof}
	
	\subsection{The $\boldsymbol{p}$-adic duality theorem for FSMPs}
	\label{subsec:main results}
	In this subsection, we state the main results. Let $\mathbf{k}$ be an index and $\mathbf{t}$ a tuple of $\mathrm{dep}(\mathbf{k})$ indeterminates. We define a $\boldsymbol{p}$-adically convergent series $\mathcal{L}_{\widehat{\mathcal{A}}, \mathbf{k}}^{\star}(\mathbf{t})$ with FSSMPs-coefficients by 
	\begin{equation}
	\mathcal{L}_{\widehat{\mathcal{A}}, \mathbf{k}}^{\star}(\mathbf{t}) := \sum_{i=0}^{\infty}\left( \text{\rm \pounds}_{\widehat{\mathcal{A}}, (\{1\}^i, \mathbf{k})}^{\scalebox{0.7}{\text{\cyr sh}}, \star}(\{1\}^i, \mathbf{t})-\frac{1}{2}\text{\rm \pounds}_{\widehat{\mathcal{A}}, (\{1\}^i, \mathbf{k})}^{\scalebox{0.7}{\text{\cyr sh}}, \star}(\{1\}^i, \mathbf{t}_1)\right)\boldsymbol{p}^i.
	\label{series}\end{equation}
	\begin{theorem}
		Let $w$ be a positive integer and $\mathbf{t}$ a tuple of $w$ indeterminates. Then, we have
		\begin{equation*}
		\mathcal{L}_{\widehat{\mathcal{A}}, \{1\}^w}^{\star}(\mathbf{t}) = \mathcal{L}_{\widehat{\mathcal{A}}, \{1\}^w}^{\star}(1-\mathbf{t})
		\end{equation*}
		in the ring $\widehat{\mathcal{A}}_{\mathbb{Z}[\mathbf{t}]}$.
		\label{main thm}\end{theorem}
	We will give a proof of Theorem \ref{main thm} in the next section. Since finite multiple polylogarithms in (\ref{series}) are of shuffle type, the case $\mathbf{k}=\{1\}^w$ (= Theorem \ref{main thm}) is essential. In fact, the following lemma holds:
	\begin{lemma}
		Let $\mathbf{k}=(k_1, \dots, k_m)$ be an index, $w = \mathrm{wt}(\mathbf{k})$, $\mathbf{t} = (t_1, \dots, t_m)$ a tuple of indeterminates, and $\bullet \in \{\emptyset, \star\}$. Then, we have
		\[
		\text{\rm \pounds}_{\widehat{\mathcal{A}}, \mathbf{k}}^{\scalebox{0.7}{\text{\cyr sh}}, \bullet}(\mathbf{t})=\text{\rm \pounds}_{\widehat{\mathcal{A}}, \{1\}^w}^{\scalebox{0.7}{\text{\cyr sh}}, \bullet}(\{0\}^{k_1-1}, t_1, \dots, \{0\}^{k_m-1}, t_m).
		\]
		\label{general index lemma}\end{lemma}
	\begin{proof}
		We can easily check it by the definition of the finite shuffle multiple polylogarithms.
	\end{proof}
	The main result of this paper is as follows: 
	\begin{theorem}
		Let $r$ be a positive integer, $\mathbf{k}_1, \dots, \mathbf{k}_r$ indices, and $\mathbf{t} = (t_1, \dots, t_r)$ a tuple of indeterminates. We define an index $\mathbf{k}$ to be $(\mathbf{k}_1, \dots, \mathbf{k}_r)$ and $\mathbf{k}'$ to be $(\mathbf{k}_1^{\vee}, \dots, \mathbf{k}_r^{\vee})$. Furthermore, we define $l_i$ and $l'_i$ by $l_i := \mathrm{dep}(\mathbf{k}_i)$ and $l'_i := \mathrm{dep}(\mathbf{k}_i^{\vee})$ respectively for $i=1, \dots, r$. Then, we have
		\[
		\mathcal{L}_{\widehat{\mathcal{A}}, \mathbf{k}}^{\star}(\{1\}^{l_1-1}, t_1, \dots, \{1\}^{l_r-1}, t_r) = \mathcal{L}_{\widehat{\mathcal{A}}, \mathbf{k}'}^{\star}(\{1\}^{l'_1-1}, 1-t_1, \dots, \{1\}^{l'_r-1}, 1-t_r)
		\]	
		in the ring $\widehat{\mathcal{A}}_{\mathbb{Z}[\mathbf{t}]}$.
		\label{main cor}\end{theorem}
	\begin{proof}
		We denote $\mathbf{k}_i$ and $\mathbf{k}^{\vee}_i$ as $(k_1^{(i)}, \dots, k_{l_i}^{(i)})$ and $({k'}_1^{(i)}, \dots, {k'}_{l'_i}^{(i)})$ respectively for $i=1, \dots, r$. Let $w:=\mathrm{wt}(\mathbf{k})$. Then, by Lemma \ref{general index lemma}, Theorem \ref{main thm}, and Definition \ref{Hoffman dual}, we have
		\begin{equation*}
		\begin{split}
		& \hspace{5.5mm} \mathcal{L}_{\widehat{\mathcal{A}}, \mathbf{k}}^{\star}(\{1\}^{l_1-1}, t_1, \dots, \{1\}^{l_r-1}, t_r) \\&= \mathcal{L}_{\widehat{\mathcal{A}}, \{1\}^w}^{\star}(\dots, \{0\}^{k_1^{(i)}-1}, 1, \dots, \{0\}^{k_{l_i-1}^{(i)}-1}, 1, \{0\}^{k_{l_i}^{(i)}-1}, t_i, \dots ) \\ &= \mathcal{L}_{\widehat{\mathcal{A}}, \{1\}^w}^{\star}(\dots, \{1\}^{k_1^{(i)}-1}, 0, \dots, \{1\}^{k_{l_i-1}^{(i)}-1}, 0, \{1\}^{k_{l_i}^{(i)}-1}, 1-t_i, \dots )\\ &= \mathcal{L}_{\widehat{\mathcal{A}}, \{1\}^w}^{\star}(\dots, \{0\}^{{k'}_1^{(i)}-1}, 1, \dots, \{0\}^{{k'}_{l'_i-1}^{(i)}-1}, 1, \{0\}^{{k'}_{l'_i}^{(i)}-1}, 1-t_i, \dots ) \\ &= \mathcal{L}_{\widehat{\mathcal{A}}, \mathbf{k}'}^{\star}(\{1\}^{l'_1-1}, 1-t_1, \dots, \{1\}^{l'_r-1}, 1-t_r).
		\end{split}
		\end{equation*}
	\end{proof}
	By considering the case $r=1$, we obtain Theorem \ref{dual fn eq for 1-var}. Theorem \ref{p-adic duality} is obtained by the substitution $t=1$ in the equality (\ref{1-var}).
	\section{Proof of theorem \ref{main thm}}
	\label{sec:Proof of main result}
	\begin{lemma}
		Let $p$ be a prime number and $n$ a positive integer satisfying $n < p$. Then, we have the following $p$-adic expansion$:$
		\[
		(-1)^n\binom{p-1}{n} = (-1)^{p-1}\left( 1-\frac{p}{n} \right) \sum_{i=0}^{\infty}\sum_{p-1 \geq m_1 \geq \cdots \geq m_i \geq n}\frac{p^i}{m_1 \cdots m_i}.
		\]
		\label{lemma by the referee}\end{lemma}
	\begin{proof}
		We can calculate as follows:
		\begin{equation*}
		\begin{split}
		(-1)^n\binom{p-1}{n} &= (-1)^n\binom{p-1}{p-1-n} = (-1)^n\frac{p-n}{n}\binom{p-1}{p-n} \\
		&= (-1)^{p-1}\left( 1-\frac{p}{n}\right) \prod_{m=n}^{p-1}\left( 1-\frac{p}{m} \right)^{-1} \\
		&= (-1)^{p-1}\left( 1-\frac{p}{n}\right) \prod_{m=n}^{p-1}\left( 1+\frac{p}{m}+\frac{p^2}{m^2}+\cdots \right) \\
		&= (-1)^{p-1}\left( 1-\frac{p}{n}\right) \sum_{i=0}^{\infty}\sum_{p-1 \geq m_1 \geq \cdots \geq m_i \geq n}\frac{p^i}{m_1 \cdots m_i}.
		\end{split}
		\end{equation*}
		This completes the proof of the lemma.
	\end{proof}
	\begin{proposition}
		Let $p$ be an odd prime number and $\mathbf{t} = (t_1, \dots, t_w)$ a tuple of indeterminates. Then, we have the following $p$-adic expansion$:$
		\begin{equation*}
		\begin{split} 
		&\sum_{p-1\geq n_1\geq \cdots \geq n_w\geq 1}(-1)^{n_1}\binom{p-1}{n_1}\frac{t_1^{n_1-n_2}\cdots t_{w-1}^{n_{w-1}-n_w}t_w^{n_w}}{n_1\cdots n_w}\\ &=\text{\rm \pounds}_{p-1, \{1\}^w}^{\scalebox{0.7}{\text{\cyr sh}}, \star}(\mathbf{t})+\sum_{i=1}^{\infty}\left( \text{\rm \pounds}_{p-1, \{1\}^{w+i}}^{\scalebox{0.7}{\text{\cyr sh}}, \star}(\{1\}^i, \mathbf{t})-\text{\rm \pounds}_{p-1, (\{1\}^{i-1}, 2, \{1\}^{w-1})}^{\scalebox{0.7}{\text{\cyr sh}}, \star}(\{1\}^{i-1}, \mathbf{t})\right)p^i.
		\end{split}
		\end{equation*}
		\label{key proposition}\end{proposition}
	\begin{proof}
		By Lemma \ref{lemma by the referee}, we have
		\begin{equation*}
		\begin{split} 
		&\sum_{p-1\geq n_1\geq \cdots \geq n_w\geq 1}(-1)^{n_1}\binom{p-1}{n_1}\frac{t_1^{n_1-n_2}\cdots t_{w-1}^{n_{w-1}-n_w}t_w^{n_w}}{n_1\cdots n_w}\\ 
		&= \sum_{p-1\geq n_1\geq \cdots \geq n_w\geq 1}\frac{t_1^{n_1-n_2}\cdots t_{w-1}^{n_{w-1}-n_w}t_w^{n_w}}{n_1\cdots n_w}\left( 1-\frac{p}{n_1} \right) \sum_{i=0}^{\infty}\sum_{p-1 \geq m_1 \geq \cdots \geq m_i \geq n_1}\frac{p^i}{m_1 \cdots m_i} \\
		&= \sum_{i=0}^{\infty}\sum_{p-1 \geq m_1 \geq \cdots \geq m_i \geq n_1 \geq \cdots n_w \geq 1}\left( 1-\frac{p}{n_1} \right) \frac{t_1^{n_1-n_2}\cdots t_{w-1}^{n_{w-1}-n_w}t_w^{n_w}}{m_1\cdots m_in_1\cdots n_w}p^i \\ 
		&= \text{\rm \pounds}_{p-1, \{1\}^w}^{\scalebox{0.7}{\text{\cyr sh}}, \star}(\mathbf{t})+\sum_{i=1}^{\infty}\left( \text{\rm \pounds}_{p-1, \{1\}^{w+i}}^{\scalebox{0.7}{\text{\cyr sh}}, \star}(\{1\}^i, \mathbf{t})-\text{\rm \pounds}_{p-1, (\{1\}^{i-1}, 2, \{1\}^{w-1})}^{\scalebox{0.7}{\text{\cyr sh}}, \star}(\{1\}^{i-1}, \mathbf{t})\right)p^i.
		\end{split}
		\end{equation*}
		This completes the proof of the proposition.
	\end{proof}
	\begin{proposition}
		Let $N$ and $w$ be positive integers. Then, the following polynomial identity holds in $\mathbb{Q}[t_1, \dots, t_w]:$
		\begin{equation*}
		\begin{split}
		&\sum_{N\geq n_1 \geq \cdots \geq n_w \geq 1}(-1)^{n_1}\binom{N}{n_1}\frac{t_1^{n_1-n_2}\cdots t_{w-1}^{n_{w-1}-n_w}(t_w^{n_w}-\frac{1}{2})}{n_1 \cdots n_w}\\
		&=\sum_{N\geq n_1 \geq \cdots \geq n_w \geq 1}\frac{(1-t_1)^{n_1-n_2} \cdots (1-t_{w-1})^{n_{w-1}-n_{w}} \{ (1-t_w)^{n_w}-\frac{1}{2} \}}{n_1\cdots n_w}.
		\end{split}
		\label{modified maltivar identity}\end{equation*}
		\label{symmetric identity}\end{proposition}
	\begin{proof}
		By \cite[Theorem 2.5]{SS}, we have
		\begin{equation*}
		\begin{split}
		&\sum_{N\geq n_1 \geq \cdots \geq n_w \geq 1}(-1)^{n_1}\binom{N}{n_1}\frac{t_1^{n_1-n_2}\cdots t_{w-1}^{n_{w-1}-n_w}t_w^{n_w}}{n_1 \cdots n_w}\\
		&=\sum_{N\geq n_1 \geq \cdots \geq n_w \geq 1}\frac{(1-t_1)^{n_1-n_2} \cdots (1-t_{w-1})^{n_{w-1}-n_{w}} \{ (1-t_w)^{n_w}-1 \}}{n_1\cdots n_w},
		\end{split}
		\end{equation*}
		and by the substitution $t_w=1$, we have
		\begin{equation*}
		\begin{split}
		&\sum_{N\geq n_1 \geq \cdots \geq n_w \geq 1}(-1)^{n_1}\binom{N}{n_1}\frac{t_1^{n_1-n_2}\cdots t_{w-1}^{n_{w-1}-n_w}}{n_1 \cdots n_w}\\ &=  -\sum_{N\geq n_1 \geq \cdots \geq n_w \geq 1}\frac{(1-t_1)^{n_1-n_2} \cdots (1-t_{w-1})^{n_{w-1}-n_{w}}}{n_1\cdots n_w}.
		\end{split}
		\end{equation*}
		By combining these two identities, we obtain the desired identity.
	\end{proof}
	In order to prove Theorem \ref{main thm}, it is sufficient to show the following theorem:
	\begin{theorem}
		Let $n$ and $w$ be positive integers and $\mathbf{t}$ a tuple of $w$ indeterminates. We define $\mathcal{L}_{\mathcal{A}_n, \{1\}^w}^{\star}(\mathbf{t})$ to be
		\begin{equation*}
		\mathcal{L}_{\mathcal{A}_n, \{1\}^w}^{\star}(\mathbf{t}) := \sum_{i=0}^{n-1}\left( \text{\rm \pounds}_{\mathcal{A}_n, \{1\}^{w+i}}^{\scalebox{0.7}{\text{\cyr sh}}, \star}(\{1\}^i, \mathbf{t})-\frac{1}{2}\text{\rm \pounds}_{\mathcal{A}_n, \{1\}^{w+i}}^{\scalebox{0.7}{\text{\cyr sh}}, \star}(\{1\}^i, \mathbf{t}_1)\right)\boldsymbol{p}^i.
		\end{equation*}
		Then, we have 
		\begin{equation}
		\mathcal{L}_{\mathcal{A}_n, \{1\}^w}^{\star}(\mathbf{t}) = \mathcal{L}_{\mathcal{A}_n, \{1\}^w}^{\star}(1-\mathbf{t})
		\label{A_n relation}\end{equation}
		in $\mathcal{A}_{n, \mathbb{Z}[\mathbf{t}]}$.
	\end{theorem}
	\begin{proof}
		We prove the equality (\ref{A_n relation}) by induction on $n$. By combining Proposition \ref{key proposition} with Proposition \ref{symmetric identity}, we have
		\begin{equation}
		\begin{split}
		&\text{\rm \pounds}_{\widehat{\mathcal{A}}, \{1\}^w}^{\scalebox{0.7}{\text{\cyr sh}}, \star}(\mathbf{t})-\frac{1}{2}\text{\rm \pounds}_{\widehat{\mathcal{A}}, \{1\}^w}^{\scalebox{0.7}{\text{\cyr sh}}, \star}(\mathbf{t}_1)
		+\sum_{i=1}^{\infty}\left\{ \left( \text{\rm \pounds}_{\widehat{\mathcal{A}}, \{1\}^{w+i}}^{\scalebox{0.7}{\text{\cyr sh}}, \star}(\{1\}^i, \mathbf{t})-\frac{1}{2}\text{\rm \pounds}_{\widehat{\mathcal{A}}, \{1\}^{w+i}}^{\scalebox{0.7}{\text{\cyr sh}}, \star}(\{1\}^{i}, \mathbf{t}_1) \right) \right. \\
		&\hspace{21mm} - \left. \left( \text{\rm \pounds}_{\widehat{\mathcal{A}}, (\{1\}^{i-1}, 2, \{1\}^{w-1})}^{\scalebox{0.7}{\text{\cyr sh}}, \star}(\{1\}^{i-1}, \mathbf{t})-\frac{1}{2}\text{\rm \pounds}_{\widehat{\mathcal{A}}, (\{1\}^{i-1}, 2, \{1\}^{w-1})}^{\scalebox{0.7}{\text{\cyr sh}}, \star}(\{1\}^{i-1}, \mathbf{t}_1)\right) \right\}\boldsymbol{p}^i\\
		&=\text{\rm \pounds}_{\widehat{\mathcal{A}}, \{1\}^w}^{\scalebox{0.7}{\text{\cyr sh}}, \star}(1-\mathbf{t})-\frac{1}{2}\text{\rm \pounds}_{\widehat{\mathcal{A}}, \{1\}^w}^{\scalebox{0.7}{\text{\cyr sh}}, \star}((1-\mathbf{t})_1).
		\end{split}
		\label{last}\end{equation}
		We see that the equality (\ref{A_n relation}) for $n=1$ holds by the projection $\pi_1 \colon \widehat{\mathcal{A}}_{\mathbb{Z}[\mathbf{t}]} \twoheadrightarrow  \mathcal{A}_{\mathbb{Z}[\mathbf{t}]}$. We assume that the equation (\ref{A_n relation}) for $n-1$ holds for any tuple of indeterminates.
		By the equality (\ref{last}) and the projection $\pi_n \colon \widehat{\mathcal{A}}_{\mathbb{Z}[\mathbf{t}]} \twoheadrightarrow  \mathcal{A}_{n, \mathbb{Z}[\mathbf{t}]}$, we have
		\begin{equation}
		\begin{split}
		\mathcal{L}_{\mathcal{A}_n, \{1\}^w}^{\star}(\mathbf{t})
		&= \text{\rm \pounds}_{\mathcal{A}_n, \{1\}^w}^{\scalebox{0.7}{\text{\cyr sh}}, \star}(1-\mathbf{t})-\frac{1}{2}\text{\rm \pounds}_{\mathcal{A}_n, \{1\}^w}^{\scalebox{0.7}{\text{\cyr sh}}, \star}((1-\mathbf{t})_1)\\
		&\ \left.+\sum_{i=1}^{n-1}\left( \text{\rm \pounds}_{\mathcal{A}_n, \{1\}^{w+i}}^{\scalebox{0.7}{\text{\cyr sh}}, \star}(\{1\}^{i-1}, t_0, \mathbf{t})-\frac{1}{2} \text{\rm \pounds}_{\mathcal{A}_n, \{1\}^{w+i}}^{\scalebox{0.7}{\text{\cyr sh}}, \star}(\{1\}^{i-1}, t_0, \mathbf{t}_1) \right)\boldsymbol{p}^i\right|_{t_0=0}.
		\end{split}
		\label{t_0=0}\end{equation}
		On the other hand, by the induction hypothesis, we have
		\[
		\mathcal{L}_{\mathcal{A}_{n-1}, \{1\}^{w+1}}^{\star}(t_0, \mathbf{t}) = \mathcal{L}_{\mathcal{A}_{n-1}, \{1\}^{w+1}}^{\star}(1-t_0, 1-\mathbf{t}).
		\] 
		Therefore, the right-hand side of (\ref{t_0=0}) coincides with $\mathcal{L}_{\mathcal{A}_n, \{1\}^w}^{\star}(1-\mathbf{t})$ and the equality (\ref{A_n relation}) for $n$ holds. Here, note that there exists the canonical isomorphism $\mathcal{A}_{n-1, \mathbb{Z}[\mathbf{t}]} \simeq \boldsymbol{p}\mathcal{A}_{n, \mathbb{Z}[\mathbf{t}]}$.
	\end{proof}

\end{document}